\documentclass{article} 

\usepackage[english]{babel} 
\usepackage{amssymb}
\usepackage{amsthm}
\usepackage{amsmath}
\usepackage{txfonts}
\usepackage{mathdots}
\usepackage[classicReIm]{kpfonts}
\usepackage{graphicx}
\usepackage{cite}
\usepackage{enumitem}
\usepackage{blindtext}

\title{Feedback Stabilization of Tank-Liquid System with Robustness to Surface Tension} 

\author{
Iasson Karafyllis, Filippos Vokos\\
Department of Mathematics,\\ National Technical University of Athens, \\
 Zografou Campus, 15780, Athens, Greece, \\
emails: iasonkar@central.ntua.gr, fivojean@mail.ntua.gr \and Miroslav Krstic \\
Department of Mechanical and Aerospace Eng.,  \\ University of California, 
La Jolla, San Diego, \\ CA 92093-0411, USA 
email: krstic@ucsd.edu}

\begin{document}
\maketitle
\begin{abstract}
We construct a robust stabilizing feedback law for the viscous Saint-Venant system of Partial Differential Equations (PDEs) with surface tension and without wall friction. The Saint-Venant system describes the movement of a tank which contains a viscous liquid. We assume constant contact angles between the liquid and the walls of the tank and we achieve a spill-free exponential stabilization with robustness to surface tension by using a Control Lyapunov Functional (CLF). The proposed CLF provides a  parameterized family of sets which approximate the state space from the interior. Based on the CLF, we construct a nonlinear stabilizing feedback law which ensures that the closed-loop system  converges exponentially to the desired equilibrium point in the sense of an appropriate norm.
\end{abstract}

\section{Introduction}\label{section1}
The Saint-Venant model, which was derived in \cite{A}, constitutes a significant and very influential mathematical model in fluid mechanics. It is also referred in literature  as the shallow water model. Recent extensions of the Saint-Venant model take into account various types of forces such as viscous stresses, surface tension and friction forces (see \cite{C,D,12,F,G,4}).

The feedback stabilization problem of the Saint-Venant model is a challenging problem. The dominant cases studied in the litterature include the inviscid model - which ignores forces such as viscous stresses and surface tension - and the linearized model (see \cite{Y1,a,b,c,d,e,f,X1,g,h,i,k,l,m}). In \cite{d,e,X1,l,m} the problem of the movement of an 1-D tank which contains a fluid is studied. More specifically, \cite{d,e,X1} provide controllability results for the Saint-Venant model without viscosity, without friction and without surface tension, while \cite{l} suggests a new variational formulation of Saint-Venant equations and proves the steady-state controllability of the linear approximations of several control configurations. In \cite{m} the inviscid Saint-Venant model is studied and appropriate stabilizing full-state feedback and output feedback control laws are constructed. In \cite{a,b,c,e,f,g,h,i,k} the movement of a fluid in an open channel is studied. Stabilization results are provided in \cite{Y1,a,b,f,h,i}. In \cite{a,b,h,i,k} the linearized Saint-Venant model is being used while \cite{Y1} deals with a general linear hyperbolic system which appears in Saint-Venant equations among other linear hyperbolic laws. The works \cite{c,e,f,g} study the nonlinear Saint-Venant model. In \cite{c} the feedforward control problem of general nonlinear hyperbolic systems is studied and an application using the Saint-Venant model with friction is provided. In \cite{e,f} local convergence of the state of hyperbolic systems of conservation laws is guaranteed using a strict Lypaunov function which exploits Riemann invariants. An application to the inviscid, frictonless Saint-Venant model is provided as well. The paper \cite{g} achieves regulation of the water flow and level in water-ways using the inviscid Saint-Venant model without friction and without surface tension.  

Very few studies in the literature deal with the nonlinear viscous Saint-Venant model that is used for the description of the movement of a tank which contains an incompressible, Newtonian fluid. The first work that studied the nonlinear viscous Saint-Venant model without wall friction and without surface tension was \cite{13}. In \cite{13} an appropriate nonlinear feedback law is constructed which provides semiglobal stabilization results by following a CLF methodology. The work \cite{15} extends the results obtained in \cite{13} in the case where wall friction forces are taken into account. In \cite{15} both the case of a velocity independent friction coefficient and the general case of friction coefficient are studied. A robust with respect to wall friction stabilizing feedback law is  constructed. Another study which deals with the nonlinear viscous Saint-Venant model is \cite{14}. In \cite{14} a stabilizing output-feedback control law for the viscous Saint-Venant PDE system without wall friction and without surface tension is constructed. The output-feedback control law is utilized through a functional-observer methodology and a CLF methodology.

The study of the movement of a fluid which interacts with a gas boundary and a solid boundary is inevitably intertwined with the notion of the surface tension and the notions of contact angle and wettability (see \cite{8,q}). Surface tension is crucial as it acts in the interface between liquid and gas. From a mathematical point of view surface tension is very important because it changes the order of the PDEs (it is expressed by a third order term). Contact angle is the angle at which the fluid surface intersects with a solid boundary as stated in \cite{q}, and it is a measure of wettability of the solid surface. There is a wide literature concerning the topic of contact angles (see for instance \cite{21,8,3,9,1,2,4,7}). The concept of contact angle is significant in our study because it provides an additional boundary condition.

In this paper we solve the feedback stabilization problem for a tank containing a liquid modeled by the viscous Saint-Venant system of PDEs with surface tension and without wall friction. We consider the case of constant contact angles between the liquid and the walls of the tank, as in \cite{1,2}. We utilize a specific form of the feedback law initially presented in \cite{15}, which constitutes a more general form of the feedback law in \cite{13} with robustness to surface tension. Indeed, we saw that the proposed feedback law guarantees stabilization \textit{no matter what the value of the surface tension coefficient is}. Therefore, the knowledge of the surface tension coefficient is not necessary and the feedback law is independent of the surface tension coefficient. We achieve a spill-free exponential stabilization, with robustness to surface tension. As in \cite{13,14,15} we follow a CLF methodology and we design the feedback law based on an appropriate functional, which is the CLF. The CLF determines a specific parameterized set which approximates the state space of the control problem from the interior. 

Although this work presents enough technical similarities with \cite{15}, there are some crucial differences. Firstly, in contrast with \cite{15}, the system of PDEs contains an extra term due to surface tension and does not contain a friction term. Moreover, in order for the model to be complete and for the problem to be well-posed, an additional boundary condition is used. The additional boundary condition is provided by the assumption of a constant contact angle. Here we use only one CLF while in \cite{15} two different functionals are proposed. As a consequence this work does not provide a bound for the sup-norm of the fluid velocity, as in \cite{15}, due to the absence of an appropriate functional. Here the CLF is different from the corresponding one in \cite{15}, as it contains an additional potential energy term due to the effect of the surface tension. 

This paper is organized as follows. In Section \ref{section2} the control problem is described as well as its main objective. In Section \ref{section3} we provide the intuitive ideas and the statements of the results of this work along with some auxiliary lemmas. Section \ref{section4} includes all the proofs of the results presented in Section \ref{section3}. Finally, Section \ref{section5} points out the conclusions of this work and suggests topics for future research.

\subsection*{Notation}
\begin{itemize}
\item[$*$] $\mathbb{R} _{+} =[0,+\infty )$ is the set of non-negative real numbers. 
\item[$*$] Let $S\subseteq \mathbb{R} ^{n} $ be an open set and let $A\subseteq \mathbb{R} ^{n} $ be a set such that $S\subseteq A\subseteq cl(S)$. By $C^{0} (A;\Omega )$, we denote the class of continuous functions on $A$, which take values in $\Omega \subseteq \mathbb{R} ^{m} $. By $C^{k} (A ; \Omega )$, where $k\ge 1$ is an integer, we denote the class of functions on $A\subseteq \mathbb{R} ^{n} $, which takes values in $\Omega \subseteq \mathbb{R} ^{m} $ and has continuous derivatives of order $k$. In other words, the functions of class $C^{k} (A;\Omega )$ are the functions which have continuous derivatives of order $k$ in $S=int(A)$ that can be continued continuously to all points in $\partial S\cap A$.  When $\Omega =\mathbb{R} $ then we write $C^{0} (A)$ or $C^{k} (A)$. When $I\subseteq \mathbb{R} $ is an interval and $G\in C^{1} (I)$ is a function of a single variable, $G'(h)$ denotes the derivative with respect to $h\in I$. 
\item[$*$] Let  $I\subseteq \mathbb{R} $ be an interval, let $a<b$ be given constants and let $u:I\times [a,b]\to \mathbb{R} $ be a given function. We utilize the notation $u[t]$ to denote the profile at certain $t\in I$, i.e., $(u[t])(x)=u(t,x)$ for all $x\in [a,b]$. When $u(t,x)$ is three times differentiable with respect to $x\in [a,b]$, we use the notation $u_{x} (t,x)$, $u_{xx} (t,x)$ and $u_{xxx} (t,x)$ for the first, second and third derivative of $u$ with respect to $x\in [a,b]$ respectively, i.e.,
\begin{equation*}
\begin{array}{l}
\displaystyle{u_{x} (t,x)=\frac{\partial \, u}{\partial \, x} (t,x) }, \displaystyle{ u_{xx} (t,x)=\frac{\partial ^{2} \, u}{\partial \, x^{2} } (t,x) } \textrm{ and } \displaystyle{ u_{xxx} (t,x)=\frac{\partial ^{3} \, u}{\partial \, x^{3} } (t,x) }
\end{array}
\end{equation*}

When $u(t,x)$ is differentiable with respect to $t$, we use the notation $u_{t} (t,x)$ for the derivative of $u$ with respect to $t$, i.e., 
$$\displaystyle{u_{t} (t,x) } \displaystyle{=\frac{\partial \, u}{\partial \, t} (t,x) }$$
\item[$*$] Given a set $U\subseteq \mathbb{R} ^{n} $, $\chi _{U} $ denotes the characteristic function of $U$ defined by
\begin{equation*} 
\chi _{U} (x):=\left\{
\begin{array}{c}
1 \quad \textrm{for all} \; x\in U \\
0 \quad \textrm{for all} \; x\notin U 
 \end{array}\right.  
\end{equation*}
The sign function ${\rm sgn}:\mathbb{R} \to \mathbb{R} $ is the function defined by 
\begin{equation*} 
{\rm sgn}(x):=\left\{
\begin{array}{c}
1 \quad \textrm{for} \; x>0 \\
0 \quad \textrm{for} \; x=0\\
-1\quad \textrm{for} \; x<0 
\end{array}\right.  
\end{equation*}
\item[$*$] Consider given constants $a, b$ such that $a<b$ . For $p\in [1,+\infty )$, $L^{p} (a,b)$ denotes the set of equivalence classes of Lebesgue measurable functions $u:(a,b)\to \mathbb{R} $ with
\begin{equation*}
\left\| u\right\| _{p} :=\left(\int _{a}^{b}\left|u(x)\right|^{p} dx \right)^{1/p} <+\infty .
\end{equation*} 
 $L^{\infty } (a,b)$ denotes the set of equivalence classes of Lebesgue measurable functions $u:(a,b)\to \mathbb{R} $ with 
\begin{equation*}
\left\| u\right\| _{\infty } :={\mathop{{\rm ess}\sup }\limits_{x\in (a,b)}} \left(\left|u(x)\right|\right)<+\infty .
\end{equation*}
For an integer $k\ge 1$, $H^{k} (a,b)$ denotes the Sobolev space of functions in $L^{2} (a,b)$ with all its weak derivatives up to order $k\ge 1$ in $L^{2} (a,b)$.
\end{itemize}

\section{The Control Problem} \label{section2}
We want to manipulate the motion of a tank which contains a viscous, Newtonian, incompressible liquid.  Viscosity is utilized as a gain in the controller on the difference between the boundary liquid levels and to settle a region of attraction. The tank is subject to an acceleration which we consider as the control input and obeys Newton's second law. The problem is described by the viscous Saint-Venant equations. We restrict our study to the one-dimensional (1-D) case of the model. Moreover, contrary to prior works, in this work we do not neglect the surface tension that acts on the free surface (liquid-gas interface) but we neglect friction with the tank walls.

We intend to drive asymptotically the tank to a specified position. The aforementioned goal must be achieved without liquid spillage and by having both the tank and the liquid within the tank at rest. 
The equations describing the motion of the liquid in the tank can be derived by performing mass and momentum balances (from first principles assuming that the liquid pressure is the combination of hydrostatic pressure and capillary pressure given by the Young-Laplace equation (see \cite{X2}) and by ignoring friction with the tank walls). The equations can also be derived by using approximations of the Navier-Stokes equations for the incompressible fluid (see \cite{1,2,3,5,6,10,11}; but see also \cite{9,12} for fluid equations involving capillary phenomena).

We denote by  $a(t)$ the position of the left side of the tank at time $t\ge 0$ and we consider the length of the tank to be $L>0$ (a constant). The evolution of the liquid level and of the liquid velocity is described by the following equations
\begin{gather}   
 H_{t} +(Hu)_{z} =0,  \textrm{ for }  t>0 ,\, z\in \left[a(t),a(t)+L\right] \label{GrindEQ__1_}  \\             
(Hu)_{t} +\left(Hu^{2} +\frac{1}{2} gH^{2} \right)_{z} -\sigma H\left(\frac{H_{zz} }{\left(1+H_{z}^{2} \right)^{3/2} } \right)_{z} 
 =\mu \left(Hu_{z} \right)_{z} \nonumber \\
\textrm{ for } t>0,\, z\in \left(a(t),a(t)+L\right) \label{GrindEQ__2_}                                          
\end{gather}
where $H(t,z)>0$, $u(t,z)\in \mathbb{R} $ are the liquid level and the liquid velocity, respectively, at time $t\ge 0$ and position $z\in \left[a(t),a(t)+L\right]$, while $g,\mu ,\sigma >0$ (constants) are the acceleration of gravity, the kinematic viscosity of the liquid and the ratio of the surface tension and liquid density, respectively. In some papers the term $\displaystyle{\left(\frac{H_{zz} }{\left(1+H_{z}^{2} \right)^{3/2} } \right)_{z} }$ is replaced by $H_{zzz} $ (see \cite{5,6,10,11}, but here we prefer a more accurate description of the surface tension.

The liquid velocities at the walls of the tank are equal with the tank velocity. Consequently:
\begin{equation} \label{GrindEQ__3_}
u(t,a(t))=u(t,a(t)+L)=w(t), \textrm{ for } t\ge 0                                
\end{equation}
where $w(t)=\dot{a}(t)$ is the velocity of the tank at time $t\ge 0$. Moreover, we get for the tank
\begin{equation} \label{GrindEQ__4_}
\ddot{a}(t)=-f(t), \textrm{ for } t>0
\end{equation}                                            
where $-f(t)$, the control input to the problem, is the tank acceleration. Defining the quantities
\begin{align}  
v(t,x)&:=u(t,a(t)+x)-w(t) \label{GrindEQ__5_a} \\ 
h(t,x)&:=H(t,a(t)+x) \label{GrindEQ__5_b} \\ 
\xi (t)&:=a(t)-a^{*}  \label{GrindEQ__5_c} 
\end{align} 
where $a^{*} \in \mathbb{R} $ is the position (a constant) which we want the left side of the tank to reach, we get the model:
\begin{gather}
\dot{\xi }=w, \textrm{ for } t\ge 0 \label{GrindEQ__6_}  \\                                           
\dot{w}=-f, \textrm{ for } t\ge 0 \\                                         
h_{t} +(hv)_{x} =0, \textrm{ for } t>0 ,\, x\in \left[0,L\right]\label{GrindEQ__7_} \\                                     
(hv)_{t} +\left(hv^{2} +\frac{1}{2} gh^{2} \right)_{x} -\sigma h\left(\frac{h_{xx} }{\left(1+h_{x}^{2} \right)^{3/2} } \right)_{x} 
=\mu \left(hv_{x} \right)_{x} +hf, \nonumber \\
\textrm{ for } t>0 ,\, x\in \left(0,L\right)\label{GrindEQ__8_} \\
v(t,0)=v(t,L)=0, \textrm{ for } t\ge 0 \label{GrindEQ__9_}                                            
\end{gather}

\noindent Equations \eqref{GrindEQ__7_} and \eqref{GrindEQ__9_} imply that every classical solution of \eqref{GrindEQ__6_}-\eqref{GrindEQ__9_} satisfies the following 
\begin{equation}
\frac{d}{d\, t} \left(\int _{0}^{L}h(t,x)dx \right)=0 \textrm{ for all } t>0
\end{equation}
Consequently, the total mass of the liquid $m>0$ is constant, and without loss of generality we can assume that every solution of \eqref{GrindEQ__6_}-\eqref{GrindEQ__9_} satisfies the equation
\begin{equation} \label{GrindEQ__10_} 
\int _{0}^{L}h(t,x)dx \equiv m 
\end{equation} 
Due to the nature of our problem it is important to mention that the liquid level $h(t,x)$ must be positive for all times, i.e., we must have:
\begin{equation} \label{GrindEQ__11_}
\mathop{\min }\limits_{x\in [0,L]} \left(h(t,x)\right)>0, \textrm{ for } t\ge 0                                          
\end{equation}
Contrary to prior works, model \eqref{GrindEQ__6_}-\eqref{GrindEQ__9_}, \eqref{GrindEQ__10_} is not a complete mathematical description of the system. This can be seen directly by studying the linearization of model \eqref{GrindEQ__6_}-\eqref{GrindEQ__9_}, \eqref{GrindEQ__10_} but also can be seen by studying the literature (see \cite{1,2,4,7,8} and references therein). For a complete mathematical model of the system we need two additional boundary conditions that describe the interaction between the liquid and the solid walls of the tank. There are many ways to describe the evolution of the angle of contact of a liquid with a solid boundary (see the detailed presentation in \cite{8}). In \cite{1,2}, Schweizer suggested (based on energy arguments and the fact that there might be a discrepancy between the actual microscopic and the apparent macroscopic contact angle) the use of a constant contact angle. Moreover, the assumption of a constant contact angle allows the well-posedness of the overall problem (at least for small data; see \cite{1,2,7}). The constant contact angle approach has been used extensively in the literature (see for instance \cite{4,7,21}). 

In this work, we adopt the constant contact angle approach by imposing a contact angle equal to $\pi /2$. Therefore, the model \eqref{GrindEQ__6_}-\eqref{GrindEQ__9_}, \eqref{GrindEQ__10_} is accompanied by the following boundary conditions:
\begin{equation} \label{GrindEQ__12_}
h_{x} (t,0)=h_{x} (t,L)=0, \textrm{ for } t\ge 0                                          
\end{equation}
In order to avoid  liquid spillage the following condition must be satisfied:
\begin{equation} \label{GrindEQ__13_}
\mathop{\max }\limits_{x\in [0,L]} \left(h(t,x)\right)<H_{\max }, \textrm{ for } t\ge 0 
\end{equation}
where $H_{\max } >0$ is the height of the tank walls. We consider classical solutions for the system \eqref{GrindEQ__6_}-\eqref{GrindEQ__9_}, \eqref{GrindEQ__10_}, \eqref{GrindEQ__12_}, i.e., we consider 
\\
\\
$\xi \in C^{2} \left(\mathbb{R} _{+} \right)$, $w\in C^{1} \left(\mathbb{R}_{+} \right)$, $h\in C^{1} \left([0,+\infty )\times [0,L]; \right.$ $\left. (0,+\infty )\right)$ $\cap C^{3} \left((0,+\infty ) \right.$ $\left. \times (0,L)\right)$, $v\in C^{0} ([0,+\infty )\times [0,L])$ $\cap C^{1} \left((0,+\infty )\right.$ $\left. \times [0,L]\right)$ with $v[t] $ $\in C^{2} \left((0,L)\right)$ for each $t>0$ 
\\
\\
that satisfy equations 
\eqref{GrindEQ__6_}-\eqref{GrindEQ__9_}, \eqref{GrindEQ__10_}, \eqref{GrindEQ__12_} for a given input $f\in C^{0} \left(\mathbb{R} _{+} \right)$.

For the system \eqref{GrindEQ__6_}-\eqref{GrindEQ__9_}, \eqref{GrindEQ__10_}, \eqref{GrindEQ__12_} with $f(t)\equiv 0$ (which is the open loop system), there exists a continuum of equilibrium points, i.e., the points
\begin{gather} 
h(x)\equiv h^{*}, v(x)\equiv 0, \textrm{ for } x\in \left[0,L\right] \label{GrindEQ__14_} \\                                       
\xi \in \mathbb{R}, w=0  \label{GrindEQ__15_}
\end{gather} 
where $h^{*} =m/L$.  We assume that the equilibrium points satisfy the condition \eqref{GrindEQ__13_}, i.e., $h^{*} <H_{\max } $.

We intend to construct a  robust with respect to surface tension control law of the form
\begin{equation} \label{GrindEQ__16_}
f(t)=F\left(h[t],v[t],\xi (t),w(t)\right), \textrm{ for } t>0,                                  
\end{equation}
which stabilizes the equilibrium point with $\xi =0$. In addition to that we impose the condition \eqref{GrindEQ__13_}. 

It follows from  \eqref{GrindEQ__14_}, \eqref{GrindEQ__15_} that the desired equilibrium point is not asymptotically stable for the open-loop system. Consequently the described control problem is not at all trivial.

\section{The feedback law} \label{section3}
\subsection{The Control Lyapunov Functional (CLF)}
We define the set $S\subset \mathbb{R} ^{2} \times \left(C^{0} \left([0,L]\right)\right)^{2} $ as follows:
\begin{equation} \label{GrindEQ__17_} 
(\xi ,w,h,v)\in S\; \Leftrightarrow \left\{\begin{array}{c} 
{h\in C^{0} \left([0,L];(0,+\infty )\right)\cap H^{1} (0,L)} \\ {v\in C^{0} \left([0,L]\right)} \\ {\displaystyle{\int _{0}^{L}h(x)dx =m }} \\ {(\xi ,w)\in \mathbb{R} ^{2} ,v(0)=v(L)=0} 
\end{array}\right.  
\end{equation} 
The above definition guarantees that every $(\xi ,w,h,v)\in S $ satisfies \eqref{GrindEQ__9_} and \eqref{GrindEQ__10_}. 
In addition to that, we define the following functionals for all $(\xi ,w,h,v)\in S$:

\begin{align} \label{GrindEQ__18_}
V(\xi ,w,h,v):=\delta E(h,v)+W(h,v)+\frac{qk^{2} }{2} \xi ^{2} +\frac{q}{2} \left(w+k\xi \right)^{2}    
\end{align}
\begin{gather} 
E(h,v):=\frac{1}{2} \int _{0}^{L}h(x)v^{2} (x)dx +\frac{g}{2} \left\| h-h^{*} \chi _{[0,L]} \right\| _{2}^{2} \nonumber \\
+\sigma \int _{0}^{L}\left(\sqrt{1+\left(h'(x)\right)^{2} } -1\right)dx \label{GrindEQ__19_} 
\end{gather} 
\begin{gather} 
W(h,v):=\frac{1}{2} \int _{0}^{L}h^{-1} (x)\left(h(x)v(x)+\mu h'(x)\right)^{2} dx+\frac{g}{2} \left\| h-h^{*} \chi _{[0,L]} \right\| _{2}^{2} \nonumber \\
+\sigma \int _{0}^{L}\Bigg(\sqrt{1+\left(h'(x)\right)^{2} } -1\Bigg)dx  \label{GrindEQ__20_} 
\end{gather}  
where $k,q>0$ are position error and velocity gains (to be selected) respectively, $\delta >0$ and $h^{*} =m/L$. In particular:
\\
\begin{itemize}
\item  the functional $E$ is the mechanical energy of the liquid within the tank as it is the sum of the potential energy
$$\displaystyle{\frac{g}{2} } \left\| h-h^{*} \chi _{[0,L]} \right\| _{2}^{2} +\sigma \int _{0}^{L}\left(\sqrt{1+\left(h'(x)\right)^{2} } -1\right)dx $$ 
and the kinetic energy
$$ \displaystyle{\frac{1}{2} \int _{0}^{L}h(x)v^{2} (x) dx} $$ 
of the liquid. It should be noticed that there is no contribution to the mechanical energy of the tank-liquid interface which allows to give the interpretation that the boundary condition \eqref{GrindEQ__12_} (a constant contact angle) is a result of the absence of interaction between liquid and solid.  
\item  the functional $W$ is a kind of mechanical energy of the liquid within the tank and has been used extensively in the literature of isentropic, compressible liquid flow (see \cite{17,18,19,20}) as well as in \cite{13,14,15}.
\end{itemize} 
\quad \vspace{-0.5em}

The functional $V(\xi ,w,h,v)$ defined by \eqref{GrindEQ__18_} will be utilized as a CLF for the system, and for the derivation of useful bounds for the function $h$ as guaranteed by the following lemma.
\newtheorem{lemma}{Lemma}
\begin{lemma} \label{lemma1}
Let constants $q,k,\delta >0$ be given, and define the increasing function $G\in C^{0} (\mathbb{R})\cap C^{1} ((-\infty ,0)\cup (0,+\infty ))$ as follows
\begin{equation} \label{GrindEQ__21_} 
G(h):=\left\{
\begin{array}{c}
{\rm sgn}\left(h-h^{*} \right) 
\displaystyle{ \left(\frac{2}{3} h\sqrt{h} -2h^{*} \sqrt{h} +\frac{4}{3} h^{*} \sqrt{h^{*} } \right)\quad} {\rm for} \; h>0 \\
\displaystyle{-\frac{4}{3} h^{*} \sqrt{h^{*} } +h}\quad {\rm for} \; h\le 0 
 \end{array}\right.  
\end{equation} 
Denote by $G^{-1} :\mathbb{R} \to \mathbb{R} $ the inverse function of $G$ and define the constant
\begin{equation} \label{GrindEQ__22_} 
c:=\frac{1}{\mu \sqrt{\delta g} }  
\end{equation} 
Then for every $(\xi ,w,h,v)\in S $, the following inequality holds:
\begin{equation} \label{GrindEQ__23_}
\begin{array}{l}
Q_{1} \left(V(\xi ,w,h,v)\right)\le h(x)\le Q_{2} \left(V(\xi ,w,h,v)\right), \textrm{ for all } x\in [0,L],                                                      
\end{array}
\end{equation}
where the functions $Q_{i} :\mathbb{R} _{+} \to \mathbb{R} $ ($i=1,2$) are defined as follows for all $s\ge 0$:
\begin{gather}   
Q_{1} (s):=\max \left(G^{-1} \left(-cs\right),N_{1}(s), N_{2}(s) \right) \label{GrindEQ__24_a} \\ 
Q_{2} (s):=\min \left(G^{-1} \left(cs\right),P_{1}(s), P_{2}(s) \right) \label{GrindEQ__24_b}  
\end{gather}
with the functions $N_{i} :\mathbb{R} _{+} \to \mathbb{R} $ ($i=1,2$) and $P_{i} :\mathbb{R} _{+} \to \mathbb{R} $ ($i=1,2$) defined by the following expressions for all $s\ge 0$:
\begin{gather}
\displaystyle{N_{1}(s):=h^{*} -\sqrt{\frac{2m\left(1+\delta \right)s}{\delta \mu ^{2} } } }, \\
\displaystyle{N_{2}(s):= h^{*} -\sqrt{\left(\frac{s}{\sigma (\delta +1)} +L\right)^{2} -L^{2} } } , \vspace{0.4em}\\
\displaystyle{P_{1}(s):=h^{*} +\sqrt{\frac{2m\left(1+\delta \right)s}{\delta \mu ^{2} } } }, \\
\displaystyle{P_{2}(s):=h^{*} +\sqrt{\left(\frac{s}{\sigma (\delta +1)} +L\right)^{2} -L^{2} } }
\end{gather}
\end{lemma}
\textbf{Remark 1.} \label{remark1}
It follows from \eqref{GrindEQ__21_}, \eqref{GrindEQ__22_}, \eqref{GrindEQ__24_a} and the fact that $h^{*} =m/L$ that 
$Q_{1} \left(V(\xi ,w,h,v)\right)>0$ when 
\begin{equation}
V(\xi ,w,h,v)<\max  \left(\theta_{1},\theta_{2}, \theta_{3} \right) 
\end{equation}
with
\begin{equation*}
\begin{array}{l}
\displaystyle{\theta_{1}:=\frac{4}{3} \mu h^{*} \sqrt{\delta gh^{*} } \textrm{, } \theta_{2}:= \frac{\mu ^{2} h^{*} \delta }{2L\left(1+\delta \right)} } \quad \textrm{ and } \vspace{0.4em} \\
\theta_{3}:= \sigma \left(\delta +1\right)\left(\sqrt{\left(h^{*} \right)^{2} +L^{2} } -L\right)
\end{array}
\end{equation*}
Definitions \eqref{GrindEQ__24_a} and \eqref{GrindEQ__24_b} imply that \textit{$Q_{2} :\mathbb{R} _{+} \to \mathbb{R} $} is an increasing function while \textit{$Q_{1} :\mathbb{R} _{+} \to \mathbb{R} $} is a decreasing function.

It is important to mention that Lemma \ref{lemma1} is more general than Lemma 1 in \cite{13} and Lemma 1 in \cite{15}. Lemma 1 in \cite{13} can be applied only for the case $\delta =1$ and $\sigma =0$, while Lemma 1 in \cite{15} can be applied only for the case $\sigma =0$. Here Lemma \ref{lemma1} can be applied for all $\delta >0$ and $\sigma \ge 0$.
\subsection{The state space}
As in \cite{13,14,15} the state space will be appropriately defined in order to exclude states of the set $S$ defined by \eqref{GrindEQ__17_} that violate the condition \eqref{GrindEQ__13_}, i.e, the states that cause liquid spillage.
We define the following 
\begin{equation} \label{GrindEQ__25_} 
X:=\left\{\, (\xi ,w,h,v)\in S\, :\, {\mathop{\max }\limits_{x\in [0,L]}} \left(h(x)\right)<H_{\max } \, \right\} 
\end{equation} 
\begin{equation} \label{GrindEQ__26_} 
R:=\frac{2\mu \sqrt{\delta gh^{*} } }{3} \left(H_{\max } -h^{*} \right)\min \left(\zeta _{1} ,\zeta _{2} \right) 
\end{equation} 
where
\begin{gather} 
\displaystyle{\zeta _{1} :=\max \left(\Gamma_{1}, \Gamma_{2}, \Gamma_{3} \right ) } \textrm{ and } \vspace{0.5em} \\
\displaystyle{\zeta _{2} := \frac{h^{*} }{H_{\max } -h^{*} } \max \left(2,\Delta_{1},\Delta_{2} \right)}
\end{gather}
with $\Gamma_{1}, \Gamma_{2}, \Gamma_{3}, \Delta_{1}$ and $\Delta_{2}$ defined as follows:
\begin{gather}
\Gamma_{1}:=\sqrt{\frac{H_{\max } }{h^{*} } } -\frac{2\sqrt{h^{*} } }{\sqrt{H_{\max } } +\sqrt{h^{*} } },   \vspace{6em} \\
\Gamma_{2}:=\frac{3\mu \sqrt{\delta } \left(H_{\max } -h^{*} \right)}{4m\left(1+\delta \right)\sqrt{gh^{*} } } , \vspace{6em}\\
\Gamma_{3}:= \frac{3\sigma (\delta +1)\left(\sqrt{L^{2} +\left(H_{\max } -h^{*} \right)^{2} } -L\right)}{2\mu \sqrt{\delta gh^{*} } \left(H_{\max } -h^{*} \right)} , \vspace{6em}\\
\displaystyle{\Delta_{1}:=\frac{3\mu \sqrt{\delta } }{4L\sqrt{gh^{*} } \left(1+\delta \right)}, } \vspace{6em} \\
\Delta_{2}:=\frac{3\sigma (\delta +1)\sqrt{h^{*} } }{2\mu \sqrt{\delta g} \left(\sqrt{\left(h^{*} \right)^{2} +L^{2} } +L\right)} 
\end{gather}
 
The aforementioned definition \eqref{GrindEQ__26_}, the fact that $h^{*} <H_{\max } $ and Lemma \ref{lemma1} imply  for all $(\xi ,w,h,v)\in S$ with $V(\xi ,w,h,v)<R$ the following
\begin{equation} \label{GrindEQ__27_}
\begin{array}{l}
0<Q_{1} \left(V(\xi ,w,h,v)\right)\le h(x)\le Q_{2} \left(V(\xi ,w,h,v)\right) 
<H_{\max }, \textrm{ for all } x\in [0,L] \vspace{0.4em}
\end{array}
\end{equation}
Consequently, the conditions \eqref{GrindEQ__13_} for avoiding liquid spillage are satisfied when $(\xi ,w,h,v)\in S$ with $V(\xi ,w,h,v)<R$. 

The set $X$ defined by \eqref{GrindEQ__25_} is the state space of system \eqref{GrindEQ__6_}-\eqref{GrindEQ__9_}, \eqref{GrindEQ__10_}, \eqref{GrindEQ__12_}. In particular, we consider as state space the metric space $X\subset \mathbb{R} ^{2} \times H^{1} \left(0,L\right)\times L^{2} \left(0,L\right)$ with metric induced by the norm of the underlying normed linear space $\mathbb{R} ^{2} \times H^{1} \left(0,L\right)\times L^{2} \left(0,L\right)$, i.e., 
\begin{equation} \label{GrindEQ__28_} 
\left\| (\xi ,w,h,v)\right\| _{X} =\left(\xi ^{2} +w^{2} +\left\| h\right\| _{2}^{2} +\left\| h'\right\| _{2}^{2} +\left\| v\right\| _{2}^{2} \right)^{1/2}  
\end{equation}
However, we need to approximate the state space from its interior by using certain parameterized sets that allow us to obtain useful estimates. We define 
\begin{align} \label{GrindEQ__29_}
X_{V} \left(r\right):=\left\{\, (\xi ,w,h,v)\in S\, :\, V(\xi ,w,h,v)\le r\, \right\}, \textrm{ for } r\ge 0  
\end{align}
Inequalities \eqref{GrindEQ__27_} imply that 
\begin{equation} \label{GrindEQ__30_}
X_{V} \left(r\right)\subseteq X, \textrm{ for all } r\in \left[0,R\right)                                     
\end{equation}
As indicated by the following proposition the set $X_{V} \left(r\right)$ for $r>0$ contains a neighborhood of  $\left(0,0,h^{*} \chi _{[0,L]} ,0\right)$ (in the topology of $X$ with metric induced by the norm $\left\| \textrm{ } \right\| _{X} $ defined by \eqref{GrindEQ__28_}).
\newtheorem{proposition}{Proposition}
\begin{proposition} \label{proposition1}
Let constants $q,k,\delta >0$ be given. Then for every $(\xi ,w,h,v)\in S$ satisfying the inequality 
\begin{equation} \label{epsilon_cond1}
\left\| (0,w,h-h^{*} \chi _{[0,L]} ,v)\right\| _{X} \le \varepsilon 
\end{equation}
\noindent for some $\varepsilon >0$ with
\begin{equation} \label{epsilon_cond2}
\varepsilon <\min \left(h^{*} ,H_{\max } -h^{*} \right) / \sqrt{L},
\end{equation}
the following inequality holds:
\begin{align} \label{GrindEQ__31_}
\hspace{-0.6em} V(\xi ,w,h,v) 
\le  C_{1} \left \| (\xi ,w,h-h^{*} \chi _{[0,L]} ,v)\right\| _{X}^{2} + C_{2} \left\| (\xi ,w,h-h^{*} \chi _{[0,L]} ,v)\right\| _{X}   
\end{align}
where 
\begin{gather}
C_{1}:= \displaystyle{ \max \Bigg ( \frac{\mu ^{2}}{h^{*} -\varepsilon \sqrt{L} }, \frac{\delta +1}{2} g,\frac{\left(\delta +2\right)H_{\max } }{2} ,q, \frac{3qk^{2} }{2} \Bigg ), } \\
\hspace{-2em} C_{2}:=\sigma (\delta +1)\sqrt{L}
\end{gather}
and $\left\| \, \cdot \, \right\| _{X} $ is defined by \eqref{GrindEQ__28_}. 
\end{proposition}
\subsection{Stabilization results}
The following theorem guarantees exponential stabilization of the state of the system \eqref{GrindEQ__6_}-\eqref{GrindEQ__9_}, \eqref{GrindEQ__10_}, \eqref{GrindEQ__12_}  by means of the nonlinear feedback law \eqref{GrindEQ__34_}.
\newtheorem{theorem}{Theorem}
\begin{theorem}[Stabilization of the Tank-Liquid System] \label{theorem1}
\quad \\
Let arbitrary constants $\omega ,k,q,\delta >0$ be given and define $R$ by means of \eqref{GrindEQ__26_}. Let arbitrary $r\in [0,R)$ be given and assume that 
\begin{equation} \label{GrindEQ__32_} 
k<q\theta (r) 
\end{equation} 
where
\begin{align} \label{GrindEQ__33_} 
\theta(r) := \frac{\omega g\mu \delta \pi ^{2} Q_{1} (r)}{g\mu \delta \pi ^{2} Q_{1} (r)+2\omega L\left(mgLH_{\max } (\delta +1)^{2} +2\mu ^{2} \delta \pi ^{2} Q_{1} (r)\right)}   
\end{align} 
where $Q_{1}$ is defined by \eqref{GrindEQ__24_a}. Then there exist constants $M,\lambda >0$ with the following property: 
\\
\\
\noindent {\normalfont\textbf{(P)}} Every classical solution of the system \eqref{GrindEQ__6_}-\eqref{GrindEQ__9_}, \eqref{GrindEQ__10_}, \eqref{GrindEQ__12_} and
\begin{gather} 
f(t)=-\omega \left((\delta +1)\int _{0}^{L}h(t,x)v(t,x)dx +\mu \left(h(t,L)-h(t,0)\right)-q\left(w(t)+k\xi (t)\right) \right ), \nonumber \\
{\rm for } \; t>0 \label{GrindEQ__34_}
\end{gather}
with $\left(\xi (0),w(0),h[0],v[0]\right)\in X_{V} (r)$, satisfies $(\xi (t),w(t),h[t],$ $v[t])$ $\in X_{V} (r)$ and the following estimate for $t\ge 0$:
\begin{align} \label{GrindEQ__35_} 
&\left\| \left(\xi (t),w(t),h[t]-h^{*} \chi _{[0,L]} ,v[t]\right)\right\| _{X} \nonumber \\ 
&\le M\exp \left(-\lambda \, t\right)\left\| \left(\xi (0),w(0),h[0]-h^{*} \chi _{[0,1]} ,v[0]\right)\right\| _{X}  
\end{align}
\end{theorem}
\textbf{Remarks on Theorem \ref{theorem1}.}

\noindent 1) The arbitrary quantities $\omega ,k,q,\delta >0$ are the control parameters. We should point out that the ratio $k/q$ must be sufficiently small due to \eqref{GrindEQ__32_}, and this is the only restriction for the control parameters.

\noindent 2) The set $X_{V} (r)$ is the set for which exponential stabilization is achieved. As indicated by Proposition \ref{proposition1}, the set $X_{V} \left(r\right)$ for $r>0$ contains a neighborhood of  $\left(0,0,h^{*} \chi _{[0,L]} ,0\right)$ (in the topology of $X$ with metric induced by the norm $\left\| \textrm{ } \right\| _{X} $ defined by \eqref{GrindEQ__28_}). The size of the set $X_{V} (r)$ depends on $r\in \left[0,R\right)$  and on $\delta ,q,k$ (recall \eqref{GrindEQ__26_} and \eqref{GrindEQ__18_}). It is straightforward to see that  the larger the parameter $q$ (or $k$)  the smaller the set $X_{V} \left(r\right)$. However, the dependence of $X_{V} (r)$ on $\delta $ (through the dependence of $R$ on $\delta $) is not clear. On the contrary it is a very complicated, non-monotonic dependence.

\noindent 3) The feedback law \eqref{GrindEQ__34_} only requires the measurement of the four following quantities:

\begin{itemize}
\item  the position of the tank denoted by $\xi (t)$, and the velocity of the tank denoted by $w(t)$,

\item  the total momentum of the liquid, i.e., the quantity $\displaystyle{\int _{0}^{L}h(t,x)v(t,x)dx} $, and

\item  the difference the liquid level at the tank walls, i.e., the quantity $h(t,L)-h(t,0)$.
\end{itemize}

\noindent It should be emphasized that the feedback law \eqref{GrindEQ__34_} does not require the measurement of the whole liquid level and liquid velocity profile whereas it is completely independent of the surface tension coefficient. 

\noindent 4) The feedback law \eqref{GrindEQ__34_} is the same feedback law that was used in \cite{13,15}. When the results in \cite{13,15} and Theorem \ref{theorem1} are taken into account then it follows that the feedback law \eqref{GrindEQ__34_} guarantees robustness with respect to surface tension as well as robustness with respect to wall friction forces. From a control point of view, this is an ideal situation: the feedback law \eqref{GrindEQ__34_} is robust with respect to all possible perturbations of the basic model, its measurement requirements are minimal and guarantees exponential stabilization of the corresponding closed-loop (nonlinear; not the linearized) system. 

\noindent 5) In contrast with \cite{15}, Theorem \ref{theorem1} does not provide an estimate for the norm $\|v_{x}[t]\|_{2}$, and consequently an estimate for the sup-norm of the fluid velocity. A topic for future research is the contruction of an appropriate CLF based on which an estimate for the norm $\|v_{x}[t]\|_{2}$ can be obtained.

\section{Proofs} \label{section4}
\begin{proof}[Proof of Lemma \ref{lemma1}]
The proof is exactly the same with the proof of Lemma 1 in \cite{15}. The only difference is that here we can obtain an additional estimate for $\left\| h-h^{*} \chi _{[0,L]} \right\| _{\infty } $. Indeed, due to the fact that the function $\varphi :\mathbb{R} _{+} \to \mathbb{R}_{+} $defined by
\begin{equation}
\varphi (s)=\sqrt{s^{2} +1} -1, \textrm{ for } s\ge 0
\end{equation}
is increasing and convex, we can use Jensen's inequality (see page 120 in \cite{16}) and get for all $h\in C^{0} \left([0,L];(0,+\infty )\right)\cap H^{1} (0,L)$ with $\displaystyle{\int _{0}^{L}h(x)dx =m}$:
\begin{gather} 
\varphi \left(\frac{1}{L} \left\| h'\right\| _{1} \right)= \varphi \left(\frac{1}{L} \int _{0}^{L}\left|h'(x)\right|dx \right) \nonumber \\
\le \frac{1}{L} \int _{0}^{L}\varphi \left(\left|h'(x)\right|\right)dx  
= \frac{1}{L} \int _{0}^{L}\left(\sqrt{\left(h'(x)\right)^{2} +1} -1\right)dx   \label{GrindEQ__42_} 
\end{gather} 
Using \eqref{GrindEQ__42_}, the inequality $\left\| h-h^{*} \chi _{[0,L]} \right\| _{\infty } \le \left\| h'\right\| _{1} $ (which is a direct consequence of the fact that there exists $x^{*} \in [0,L]$ such that $h(x^{*} )=h^{*} $; a consequence of continuity of $h$, the mean value theorem and the facts that $\displaystyle{\int _{0}^{L}h(x)dx =m}$, $h^{*} =m/L$), the fact that the function $\varphi ^{-1} :\mathbb{R} _{+} \to \mathbb{R}_{+} $ (the inverse function of $\varphi $) is increasing with $\varphi ^{-1} (s)=\sqrt{\left(s+1\right)^{2} -1} $ for $s\ge 0$ and the inequality  
\begin{equation}
\int _{0}^{L}\left(\sqrt{\left(h'(x)\right)^{2} +1} -1\right)dx \le \frac{V(\xi ,w,h,v)}{\sigma (\delta +1)}
\end{equation}  
\noindent which is a direct consequence of definitions \eqref{GrindEQ__18_}, \eqref{GrindEQ__19_}, \eqref{GrindEQ__20_}, we get for all $(\xi ,w,h,v)\in S$:
\begin{equation} \label{GrindEQ__43_} 
\left\| h-h^{*} \chi _{[0,L]} \right\| _{\infty } \le \sqrt{\left(L+\frac{V(\xi ,w,h,v)}{\sigma (\delta +1)} \right)^{2} -L^{2} }  
\end{equation} 
Using the additional estimate \eqref{GrindEQ__43_} in conjunction with the estimates shown in the proof of Lemma 1 in \cite{15} and definitions \eqref{GrindEQ__22_}, \eqref{GrindEQ__24_a} and \eqref{GrindEQ__24_b} we get  \eqref{GrindEQ__23_} . 

\noindent The proof is complete. 
\end{proof}
\begin{proof}[Proof of Proposition \ref{proposition1}]
Consider arbitrary $(\xi,$ $w, h, v)\in S$ satisfying \eqref{epsilon_cond1} and \eqref{epsilon_cond2}.
Definitions \eqref{GrindEQ__18_}, \eqref{GrindEQ__19_}, \eqref{GrindEQ__20_} and the inequalities 
\begin{align}
(h(x)v(x)+\mu h'(x))^{2} &\le 2h^{2} (x)v^{2} (x)+2\mu ^{2} \left(h'(x)\right)^{2}, \label{inequality_1} \\
\left(w+k\xi \right)^{2} &\le 2w^{2} +2k^{2} \xi ^{2} , \label{inequality_2} \\
\sqrt{1+\left(h'(x)\right)^{2} } -1 &\le \left|h'(x)\right| \label{inequality_3}
\end{align}
imply: 
\begin{align} \label{GrindEQ__44_}  
V\left(\xi ,w,h,v\right) \le \frac{\delta +2}{2} \int _{0}^{L}h(x)v^{2} (x)dx 
+\mu ^{2} \int _{0}^{L}h^{-1} (x)\left(h'(x)\right)^{2} dx \nonumber \\ 
+\frac{\delta +1}{2} g\left\| h-h^{*} \chi _{[0,L]} \right\| _{2}^{2} 
+\frac{3qk^{2} }{2} \xi ^{2} +qw^{2} +\sigma (\delta +1)\left\| h'\right\| _{1}    
\end{align}  
\noindent Following the arguments of the proof of Proposition 2.5 in \cite{15} we obtain from \eqref{GrindEQ__44_} the following:
\begin{gather}  
V\left(\xi ,w,h,v\right)  
\le \frac{\delta +2}{2} H_{\max } \left\| v\right\| _{2}^{2} +qw^{2} +\frac{3qk^{2} }{2} \xi ^{2}  \nonumber \\
+\mu ^{2} \left(h^{*} -\varepsilon \sqrt{L} \right)^{-1} \left\| h'\right\| _{2}^{2} 
 +\frac{\delta +1}{2} g\left\| h-h^{*} \chi _{[0,L]} \right\| _{2}^{2} 
+\sigma (\delta +1)\sqrt{L} \left\| h'\right\| _{2} \label{GrindEQ__47_}  
\end{gather} 
\noindent Inequality \eqref{GrindEQ__31_} is a direct consequence of \eqref{GrindEQ__47_} and definition \eqref{GrindEQ__28_}. The proof is complete.
\end{proof}

In order to give the proof of the main result of this study, we need to provide some preliminary lemmas along with their proofs.

\begin{lemma} \label{lemma2}
Every classical solution of the system \eqref{GrindEQ__6_}-\eqref{GrindEQ__9_}, \eqref{GrindEQ__10_}, \eqref{GrindEQ__12_} satisfies the following equations for all $t>0$:
\begin{align} \label{GrindEQ__36_} 
\frac{d}{dt} E(h[t],v[t])=-\mu \int _{0}^{L}h(t,x) v_{x}^{2} (t,x)dx +f(t)\int _{0}^{L}h(t,x) v(t,x)dx 
\end{align} 
\begin{gather} 
\frac{d}{dt} W(h[t],v[t]) 
=-\mu g\left\| h_{x} [t]\right\| _{2}^{2}  
-\mu \sigma \int _{0}^{L}\frac{h_{xx}^{2} (t,x)dx}{\left(1+h_{x}^{2} (t,x)\right)^{3/2} } \nonumber  \\
+f(t)\int _{0}^{L}\left(h(t,x)v(t,x)+\mu h_{x} (t,x)\right)dx \label{GrindEQ__37_}   
\end{gather} 
where $E,W$ are defined by \eqref{GrindEQ__19_}, \eqref{GrindEQ__20_}, respectively.
\end{lemma}
\begin{proof}
\noindent Due to \eqref{GrindEQ__7_} and \eqref{GrindEQ__8_} we get for $t>0$, $x\in \left(0,L\right)$: 
\begin{gather}   
v_{t} (t,x)+v(t,x)v_{x} (t,x)+gh_{x} (t,x) \nonumber \\
=\sigma h^{-1} (t,x)\left(\frac{1+h_{x}^{2} (t,x)+h(t,x)h_{xx} (t,x)}{\left(1+h_{x}^{2} (t,x)\right)^{3/2} } \right)_{x} \nonumber \\
+\mu h^{-1} (t,x)\left(h(t,x)v_{x} (t,x)\right)_{x} +f(t) \label{GrindEQ__49_}
\end{gather} 
\noindent Combining definition \eqref{GrindEQ__19_}, \eqref{GrindEQ__7_} and \eqref{GrindEQ__49_} we get for all $t>0$ the following expression for the time derivative of the functional \eqref{GrindEQ__19_} :
\begin{gather}  
\frac{d}{dt} E(h[t],v[t]) 
= -\frac{1}{2} \int _{0}^{L}(h(t,x)v(t,x))_{x} v^{2} (t,x)dx \nonumber \\
-\int _{0}^{L}h(t,x)v^{2} (t,x)v_{x} (t,x)dx  
-g\int _{0}^{L}h(t,x)v(t,x)h_{x} (t,x)dx \nonumber \\
+\sigma \int _{0}^{L}v(t,x) \left(\frac{1+h_{x}^{2} (t,x)+h(t,x)h_{xx} (t,x)}{\left(1+h_{x}^{2} (t,x)\right)^{3/2} } \right)_{x} dx \nonumber  \\ 
+\mu \int _{0}^{L}v(t,x)\left(h(t,x)v_{x} (t,x)\right)_{x} dx 
+f(t)\int _{0}^{L}h(t,x)v(t,x)dx \nonumber \\ 
-g\int _{0}^{L}(h(t,x)v(t,x))_{x} (h(t,x)-h^{*} )dx  \nonumber \\
-\sigma \int _{0}^{L}\frac{h_{x} (t,x)}{\sqrt{1+h_{x}^{2} (t,x)} } (h(t,x)v(t,x))_{xx} dx  \label{GrindEQ__50_}
\end{gather} 
\noindent Using \eqref{GrindEQ__50_}, integration by parts as in the proof of Lemma 2.11 in \cite{15}, \eqref{GrindEQ__9_}, \eqref{GrindEQ__12_} and the fact that for all $t>0$ 
\begin{align}  
&\sigma \int _{0}^{L}v(t,x)\left(\frac{1+h_{x}^{2} (t,x)+h(t,x)h_{xx} (t,x)}{\left(1+h_{x}^{2} (t,x)\right)^{3/2} } \right)_{x} dx  \nonumber \\ 
&=-\sigma \int _{0}^{L}v_{x} (t,x)\frac{1+h_{x}^{2} (t,x)+h(t,x)h_{xx} (t,x)}{\left(1+h_{x}^{2} (t,x)\right)^{3/2} } dx  \label{GrindEQ__51_a} \\
&-\sigma \int _{0}^{L}\frac{h_{x} (t,x)}{\sqrt{1+h_{x}^{2} (t,x)} } (h(t,x)v(t,x))_{xx} dx  \nonumber \\
&=\sigma \int _{0}^{L}v_{x} (t,x)\frac{1+h_{x}^{2} (t,x)+h_{xx} (t,x)h(t,x)}{\left(1+h_{x}^{2} (t,x)\right)^{3/2} } dx  \label{GrindEQ__51_}
\end{align} 
as a consequence of integration by parts as well, we obtain equation \eqref{GrindEQ__36_}.
\\
\\ 
Next we define for all $t\ge 0$ and $x\in [0,L]$:
\begin{equation} \label{GrindEQ__52_} 
\varphi (t,x):=h(t,x)v(t,x)+\mu h_{x} (t,x) 
\end{equation} 

\noindent Definition \eqref{GrindEQ__52_}, \eqref{GrindEQ__7_} and \eqref{GrindEQ__8_} imply for all $t>0$ and $x\in (0,L)$:
\begin{gather} 
\varphi _{t} (t,x)=-\Biggl (v(t,x)\varphi (t,x)+\frac{1}{2} gh^{2} (t,x) 
-\sigma \frac{1+h_{x}^{2} (t,x)+h(t,x)h_{xx} (t,x)}{\left(1+h_{x}^{2} (t,x)\right)^{3/2} } \Biggl )_{x}  \nonumber \\
+h(t,x)f(t) \label{GrindEQ__53_}
\end{gather}
Using definition \eqref{GrindEQ__20_} along with \eqref{GrindEQ__53_} and \eqref{GrindEQ__7_}, we get for all $t>0$ :
\begin{gather}  
\frac{d}{dt} W(h[t],v[t]) 
=\frac{1}{2} \int _{0}^{L}h^{-2} (t,x)\varphi ^{2} (t,x)(h(t,x)v(t,x))_{x} dx  \nonumber \\
-\int _{0}^{L}h^{-1} (t,x)\varphi (t,x) \left(\varphi (t,x)v(t,x)+\frac{1}{2} gh^{2} (t,x)\right)_{x} dx  \nonumber \\ 
+\sigma \int _{0}^{L}h^{-1} (t,x) \varphi (t,x) \left(\frac{1+h_{x}^{2} (t,x)+h(t,x)h_{xx} (t,x)}{\left(1+h_{x}^{2} (t,x)\right)^{3/2} } \right)_{x} dx \nonumber  \\
+f(t)\int _{0}^{L}\varphi (t,x)dx   
-g\int _{0}^{L}(h(t,x)-h^{*} )(h(t,x)v(t,x))_{x} dx \nonumber \\
-\sigma \int _{0}^{L}\frac{h_{x} (t,x)(h(t,x)v(t,x))_{xx} }{\sqrt{1+h_{x}^{2} (t,x)} } dx \label{GrindEQ__55_}
\end{gather} 
\noindent Using \eqref{GrindEQ__9_} and integration by parts as in proof of Lemma 2.11 in \cite{15}, we obtain from \eqref{GrindEQ__55_} and definition \eqref{GrindEQ__52_} for all $t>0$: 
\begin{gather} 
\frac{d}{dt} W(h[t],v[t]) 
=-\mu g\left\| h_{x} [t]\right\| _{2}^{2} 
+f(t)\int _{0}^{L}\left(h(t,x)v(t,x)+\mu h_{x} (t,x)\right)dx  \nonumber \\ 
+\sigma \int _{0}^{L}v(t,x) \left(\frac{1+h_{x}^{2} (t,x)+h(t,x)h_{xx} (t,x)}{\left(1+h_{x}^{2} (t,x)\right)^{3/2} } \right)_{x} dx  \nonumber \\
+\mu \sigma \int _{0}^{L}h^{-1} (t,x)h_{x} (t,x) \left(\frac{1+h_{x}^{2} (t,x)+h(t,x)h_{xx} (t,x)}{\left(1+h_{x}^{2} (t,x)\right)^{3/2} } \right)_{x} dx  \nonumber  \\
-\sigma \int _{0}^{L}\frac{h_{x} (t,x)(h(t,x)v(t,x))_{xx} }{\sqrt{1+h_{x}^{2} (t,x)} } dx \label{GrindEQ__57_}    
\end{gather} 
\noindent Using \eqref{GrindEQ__12_}, \eqref{GrindEQ__51_a},  \eqref{GrindEQ__51_} and the fact that 
\begin{align}
h(t,x)\left(\frac{h_{xx} (t,x)}{\left(1+h_{x}^{2} (t,x)\right)^{3/2} } \right)_{x} =\left(\frac{1+h_{x}^{2} (t,x)+h(t,x)h_{xx} (t,x)}{\left(1+h_{x}^{2} (t,x)\right)^{3/2} } \right)_{x}  
\end{align}
\noindent we obtain from \eqref{GrindEQ__57_} equation \eqref{GrindEQ__37_} for all $t>0$. The proof is complete. 
\end{proof}
\begin{lemma} \label{lemma3}
Let constants $q,k,\delta >0$ be given. Then there exists a non-decreasing function $\Lambda :[0,R)\to (0,+\infty )$, where $R>0$ is defined by \eqref{GrindEQ__26_} such that for every $(\xi ,w,h,v)\in X$ with $v\in H^{1} (0,L)$, $h\in H^{2} (0,L)$ and $V(\xi ,w,h,v)<R$, the following inequality holds:
\begin{gather} 
\frac{V(\xi ,w,h,v)}{\Lambda (V(\xi ,w,h,v))} 
 \le  \left\| h'\right\| _{2}^{2} +\int _{0}^{L}\frac{\left(h''(x)\right)^{2} }{\left(1+\left(h'(x)\right)^{2} \right)^{3/2} } dx \nonumber  \\
+\int _{0}^{L}h(x)\left(v'(x)\right)^{2} dx +\xi ^{2} +\left(w+k\xi \right)^{2} \label{GrindEQ__38_}   
\end{gather} 
\end{lemma}
\begin{proof}
Let arbitrary $(\xi ,w,h,v)\in X$ with $v\in H^{1} (0,L)$, $h\in H^{2} (0,L)$ and $V(\xi ,w,h,v)<R$ be given. Using the same arguments as in the proof of Lemma 2.12 in \cite{15}
and the fact that 
\begin{equation} \label{GrindEQ__61_} 
\int _{0}^{L}\left(\sqrt{1+\left(h'(x)\right)^{2} } -1\right)dx \le \left\| h'\right\| _{2}^{2}  
\end{equation} 
we obtain the following estimate:
\begin{gather}  
V(\xi ,w,h,v) 
\le  \frac{L^{2} \left(\delta +2\right)Q_{2} (V(\xi ,w,h,v))}{2\pi ^{2} Q_{1} (V(\xi ,w,h,v))} \int _{0}^{L}h(x)\left(v'(x)\right)^{2} dx \nonumber  \\
+\left(\frac{(\delta +1)\left(gL^{2} +2\sigma \right)}{2} +\frac{\mu ^{2} }{Q_{1} \left(V(\xi ,w,h,v)\right)} \right) \left\| h'\right\| _{2}^{2} 
+\frac{qk^{2} }{2} \xi ^{2} +\frac{q}{2} \left(w+k\xi \right)^{2} \nonumber \\
\le \Lambda (V(\xi ,w,h,v)) \nonumber \\
\times \bigg(\left\| h'\right\| _{2}^{2} +\int _{0}^{L}\hspace{-0.7em}\frac{\left(h''(x)\right)^{2} }{\left(1+\left(h'(x)\right)^{2} \right)^{3/2} } dx  
+\int _{0}^{L}h(x)\left(v'(x)\right)^{2} dx+\xi ^{2} +\left(w+k\xi \right)^{2} \bigg ) \label{GrindEQ__62_}
\end{gather}
\noindent where
\begin{equation} \label{GrindEQ__63_}
\Lambda (s):=\frac{1}{2} \max \left(\kappa _{1} +\frac{2\mu ^{2} }{Q_{1} \left(s\right)} ,\frac{\kappa _{2} Q_{2} (s)}{Q_{1} (s)} ,\kappa _{3} \right), \textrm{ for } s\in [0,R)
\end{equation} 
\noindent with $\kappa _{1} :=(\delta +1)\left(gL^{2} +2\sigma \right)$, $\displaystyle{\kappa _{2} :=L^{2} \left(\delta +2\right)/ \pi ^{2} }$ and $\kappa _{3} :=q\max (1,k^{2} )$. Definition \eqref{GrindEQ__63_} and the fact that $Q_{2} :\mathbb{R} _{+} \to \mathbb{R} $ is an increasing function and $Q_{1} :\mathbb{R}_{+} \to \mathbb{R} $ is a decreasing function imply that $\Lambda :[0,R)\to (0,+\infty )$ is a non-decreasing function. Inequality \eqref{GrindEQ__38_} holds as a direct consequence of \eqref{GrindEQ__62_}. The proof is complete. 
\end{proof}
\begin{lemma} \label{lemma4}
Let constants $q,k,\delta >0$ be given. Then there exist non-decreasing functions $G_{i} :[0,R)\to (0,+\infty )$, $i=1,2$, where $R>0$ is defined by \eqref{GrindEQ__26_}, such that for every $(\xi ,w,h,v)\in X$ with $V(\xi ,w,h,v)<R$, the following inequalities hold:
\begin{align} \label{GrindEQ__39_}  
\left\| (\xi ,w,h-h^{*} \chi _{[0,L]} ,v)\right\| _{X}^{2} 
\le V(\xi ,w,h,v)G_{1} \left(V(\xi ,w,h,v)\right) 
\end{align}
\begin{align} \label{GrindEQ__40_} 
\frac{V(\xi ,w,h,v)}{G_{2} \left(V(\xi ,w,h,v)\right)} \le \left\| (\xi ,w,h-h^{*} \chi _{[0,L]} ,v)\right\| _{X}^{2}  
\end{align}
where $\left\| \, \cdot \, \right\| _{X} $ is defined by \eqref{GrindEQ__28_}. 
\end{lemma}
\begin{proof}
Let arbitrary $\left(\xi ,w,h,v\right)\in X$ with $V(\xi ,w,h,v)<R$ be given. Using definitions \eqref{GrindEQ__18_}, \eqref{GrindEQ__19_}, \eqref{GrindEQ__20_}, inequalities \eqref{inequality_1}, \eqref{inequality_2}, the inequality
\begin{align}
\sqrt{1+\left(h'(x)\right)^{2} } \le 1+\left(h'(x)\right)^{2}
\end{align} 
and \eqref{GrindEQ__27_} we obtain
\begin{gather}  
V(\xi ,w,h,v) 
\le \frac{\delta +2}{2} H_{\max } \left\| v\right\| _{2}^{2}  +\frac{\delta +1}{2} g\left\| h-h^{*} \chi _{[0,L]} \right\| _{2}^{2}    \nonumber \\ 
+\left(\frac{\mu ^{2} }{Q_{1} \left(V\left(\xi ,w,h,v\right)\right)} +\sigma \left(\delta +1\right)\right)\left\| h'\right\| _{2}^{2}  
+\frac{3qk^{2} }{2} \xi ^{2} +qw^{2} \label{GrindEQ__64_} 
\end{gather} 
\noindent Inequality \eqref{GrindEQ__64_} implies inequality \eqref{GrindEQ__40_} with
\begin{gather}
G_{2} \left(s\right):=\max \left(\frac{\delta +2}{2} H_{\max } ,\frac{\delta +1}{2} g,\frac{\mu ^{2} }{Q_{1} \left(s\right)} +\sigma \left(\delta +1\right),\frac{3qk^{2} }{2} ,q\right), \nonumber \\
\hspace{-2em} \textrm{for } s\in [0,R) 
\end{gather}
\noindent The fact that $Q_{1} :\mathbb{R} _{+} \to \mathbb{R} $ is a decreasing function and the above definition imply that $G_{2} :[0,R)\to (0,+\infty )$ is a non-decreasing function. 

The proof of inequality \eqref{GrindEQ__39_} is exactly the same with the proof of Lemma 4 in \cite{15}. The proof is complete. 
\end{proof}
\begin{lemma} \label{lemma5}
Let constants $\omega ,k,q,\delta >0$ and $r\in [0,R)$ be given, where $R>0$ is defined by \eqref{GrindEQ__26_}. Then every classical solution of the system \eqref{GrindEQ__6_}-\eqref{GrindEQ__9_}, \eqref{GrindEQ__10_}, \eqref{GrindEQ__12_} and \eqref{GrindEQ__34_} satisfies the following inequality for all $t>0$ for which $V(\xi (t),w(t),h[t],v[t])<R$:
\begin{gather} 
\frac{d}{dt} V(\xi (t),w(t),h[t],v[t]) 
\le  -\frac{3\mu g}{4} \left\| h_{x} [t]\right\| _{2}^{2} -qk^{3} \xi ^{2} (t) \nonumber \\
-\frac{\mu \delta }{2H_{\max } } \left(2H_{\max } -Q_{1} (r)\frac{Q_{2} \left(V(t)\right)}{Q_{1} \left(V(t)\right)} \right) \int _{0}^{L}  h(t,x)v_{x}^{2} (t,x)dx \nonumber \\ 
-\mu \sigma \int _{0}^{L}\frac{h_{xx}^{2} (t,x)}{\left(1+h_{x}^{2} (t,x)\right)^{3/2} } dx -q\left(q\theta (r)-k\right)\left(w(t)+k\xi (t)\right)^{2} \label{GrindEQ__41_}   
\end{gather} 
where $V(t)=V(\xi (t),w(t),h[t],v[t])$, $\theta (r)$ is defined by \eqref{GrindEQ__33_} and $Q_{i} :\mathbb{R} _{+} \to \mathbb{R} $ ($i=1,2$) are the functions defined by \eqref{GrindEQ__24_a} and \eqref{GrindEQ__24_b}. 
\end{lemma}
\begin{proof}
Let $\omega ,k,q,\delta >0$ be given constants and let $r\in [0,R)$ be a constant, where $R>0$ is defined by \eqref{GrindEQ__26_}. In addition to that we consider a classical solution of the system \eqref{GrindEQ__6_}-\eqref{GrindEQ__9_}, \eqref{GrindEQ__10_}, \eqref{GrindEQ__12_} and \eqref{GrindEQ__34_} at a time $t>0$ for which $V(\xi (t),w(t),h[t],v[t])<R$. Using Lemma \ref{lemma2}, \eqref{GrindEQ__36_}, \eqref{GrindEQ__37_} and definition \eqref{GrindEQ__18_} and by following the same procedure as in the proof of Lemma 2.14 in \cite{15} by assuming zero friction coefficient,  we establish the following inequality:
\begin{gather}  
\frac{d}{dt} V(\xi (t),w(t),h[t],v[t]) 
\le  -\frac{3\mu g}{4} \left\| h_{x} [t]\right\| _{2}^{2} 
-\mu \delta \int _{0}^{L}h(t,x) v_{x}^{2} (t,x)dx \nonumber  \\ 
-\mu \sigma \int _{0}^{L}\frac{h_{xx}^{2} (t,x)}{\left(1+h_{x}^{2} (t,x)\right)^{3/2} } dx 
-q\left(q\theta (r)-k\right)\left(w(t)+k\xi (t)\right)^{2}   
-qk^{3} \xi ^{2} (t) \nonumber \\
+\frac{\mu \delta \pi ^{2} Q_{1} (r)}{2L^{2} H_{\max } } \int _{0}^{L}h(t,x)v^{2} (t,x)dx \label{GrindEQ__69_} 
\end{gather} 
Since $v(t,0)=v(t,L)=0$ (recall \eqref{GrindEQ__9_}), by virtue of Wirtinger's inequality and \eqref{GrindEQ__27_}, we get:
\begin{align} \label{GrindEQ__70_} 
\left\| v[t]\right\| _{2}^{2} \le \frac{L^{2} }{\pi ^{2} } \left\| v_{x} [t]\right\| _{2}^{2} \le \frac{L^{2} }{\pi ^{2} Q_{1} \left(V(t)\right)} \int _{0}^{L}h(t,x)v_{x}^{2} (t,x)dx   
\end{align} 
Combining \eqref{GrindEQ__27_}, \eqref{GrindEQ__69_} and \eqref{GrindEQ__70_}, we obtain \eqref{GrindEQ__41_}. The proof is complete. 
\end{proof}
We can now present the proof of Theorem \ref{theorem1}.
\begin{proof}[Proof of Theorem \ref{theorem1}]
Let constants $\omega ,q,k,\delta >0$ satisfying \eqref{GrindEQ__32_}. Let constant $r\in [0,R)$ be given. Consider a classical solution of the system \eqref{GrindEQ__6_}-\eqref{GrindEQ__9_}, \eqref{GrindEQ__10_}, \eqref{GrindEQ__12_} with \eqref{GrindEQ__34_} that satisfies $V\left(\xi (0),w(0),h[0],v[0]\right)\le r$. Let $\overline{r}\in \left(r,R\right)$ be a constant that satisfies:
\begin{equation} \label{GrindEQ__71_} 
\frac{Q_{2} \left(\overline{r}\right)}{Q_{1} \left(\overline{r}\right)} <\frac{2H_{\max } }{Q_{1} (r)}  
\end{equation} 
The existence of $\bar{r}\in \left(r,R\right)$  is a direct consequence of the continuity of the functions involved in \eqref{GrindEQ__71_}. 

Due to \eqref{GrindEQ__32_}, Lemma \ref{lemma5},  \eqref{GrindEQ__41_} and \eqref{GrindEQ__71_} the following implication is true:
\begin{equation} \label{GrindEQ__72_}
\begin{array}{l}
\textrm{If } t>0 \textrm{ and } V\left(\xi (t),w(t),h[t],v[t]\right)\le \overline{r} \textrm{ then } 
\displaystyle{\frac{d}{d\, t} V\left(\xi (t),w(t),h[t],v[t]\right)\le 0 }
\end{array} 
\end{equation} 
A contradiction argument as in the proof of Theorem 2.6 in \cite{15} implies that $V\left(\xi (t),w(t),\right.$ $\left. h[t],v[t]\right)\le \overline{r}$ for all $t\ge 0$. 

Implication \eqref{GrindEQ__72_} and the fact $V\left(\xi (t),w(t),h[t],v[t]\right)\le \overline{r}$ for all $t\ge 0$ imply that 
\begin{equation}
\frac{d}{d\, t} V\left(\xi (t),w(t),h[t],v[t]\right)\le 0 \textrm{ for all } t>0
\end{equation} 
Due to the above and the continuity of the mapping $t\to V(\xi (t),w(t),h[t],$ $v[t])$, we get that 
\begin{align} \label{GrindEQ__73_}
V(\xi (t),w(t),h[t],v[t]) \le V(\xi (0),w(0),h[0],v[0]) \le r<R, \textrm{for all } t\ge 0 
\end{align}
Consequently, $\left(\xi (t),w(t),h[t],v[t]\right)\in X_{V} (r)$ for all $t\ge 0$ (recall \eqref{GrindEQ__29_}). Using \eqref{GrindEQ__73_} and Lemma \ref{lemma5}, we conclude that \eqref{GrindEQ__41_} holds for all $t>0$. Using \eqref{GrindEQ__73_}, \eqref{GrindEQ__41_} and the fact that $Q_{2} :\mathbb{R} _{+} \to \mathbb{R} $ is an increasing function while $Q_{1} :\mathbb{R}_{+} \to \mathbb{R} $ is a decreasing function, we obtain the following estimate for $t>0$
\begin{gather} 
\frac{d}{dt} V(\xi (t),w(t),h[t],v[t])  \nonumber \\
\le -\beta (r) \biggl (\left\| h_{x} [t]\right\| _{2}^{2} +\int _{0}^{L}h(t,x)v_{x}^{2} (t,x)dx  
+\int _{0}^{L}\frac{h_{xx}^{2} (t,x)}{\left(1+h_{x}^{2} (t,x)\right)^{3/2} } dx \nonumber \\
+\xi ^{2} (t) 
+\left(w(t)+k\xi (t)\right)^{2}  \biggl  )  \label{GrindEQ__74_}
\end{gather}
where
\begin{equation} \label{GrindEQ__75_} 
\begin{array}{l}
\beta (r):=
\displaystyle{\min \biggl (\frac{3\mu g}{4} ,\frac{\mu \delta \left(2H_{\max } -Q_{2} \left(r\right)\right)}{2H_{\max } } , qk^{3} ,q\left(q\theta (r)-k\right),\mu \sigma \biggl )} 
\end{array}
\end{equation} 
Notice that \eqref{GrindEQ__32_} and the fact that $r\in [0,R)$ in conjunction with definitions \eqref{GrindEQ__24_b}, \eqref{GrindEQ__26_}, \eqref{GrindEQ__74_} imply that $\beta (r)>0$. It follows from Lemma \ref{lemma3}, \eqref{GrindEQ__38_}, the continuity of the mapping $t\to V(\xi (t),w(t),h[t],v[t])$, (recall that $v\in C^{0} \left(\mathbb{R} _{+} \right.$ $\left. ;H^{1} \left(0,L\right)\right)$, $h\in C^{1} \left(\mathbb{R}_{+} \times [0,L];(0,+\infty )\right)$ and $v\in C^{0} \left(\mathbb{R}_{+} \times [0,L]\right)$), estimates \eqref{GrindEQ__73_}, \eqref{GrindEQ__74_}, Lemma \ref{lemma4}, \eqref{GrindEQ__39_} and \eqref{GrindEQ__40_} that the following estimate holds for all $t\ge 0$:
\begin{gather}  
\left\| (\xi (t),w(t),h[t]-h^{*} \chi _{[0,L]} ,v[t])\right\| _{X}^{2}  
 \nonumber \\ 
\le \Omega (r)  
\exp \left(-\frac{\beta (r)\, t}{\Lambda \left(r\right)} \right) \left\| (\xi (0),w(0),h[0]-h^{*} \chi _{[0,L]} ,v[0])\right\| _{X}^{2} \label{GrindEQ__79_}   
\end{gather} 
with
\begin{equation} \label{GrindEQ__80_} 
\Omega (r):=G_{1} \left(r\right)G_{2} \left(r\right) 
\end{equation} 
where $\Lambda $ is the non-decreasing function involved in \eqref{GrindEQ__38_} and $G_{i} :\left[0,R\right)\to \left(0,+\infty \right)$ $(i=1,2)$ are the non-decreasing functions involved in \eqref{GrindEQ__39_}, \eqref{GrindEQ__40_}. Estimate \eqref{GrindEQ__35_} with $M=\sqrt{\Omega (r)} $ and $\displaystyle{\lambda =\frac{\beta (r)}{2\Lambda (r)} }$ is a consequence of estimate \eqref{GrindEQ__79_}. The proof is complete. 
\end{proof}

\section{Concluding Remarks} \label{section5}
In this work we managed to show that the robust with respect to wall friction nonlinear feedback law proposed in \cite{15} provides also robust stabilization results with respect to surface tension. This shows even more the significance of the CLFs as stabilizing tools for the infinite-dimensional case of systems described by PDEs  and illustrates the fact that robustness is inherent in the CLF methodology.

The present study deals with the case of viscous Saint-Venant system with surface tension and without wall friction. It is of interest to study the more challenging problem of the viscous Saint-Venant system with surface tension and with wall friction as well as the construction of an additional functional which provides a bound for the sup-norm of the fluid velocity. In addition to that, other topics for future research are the study of existence and uniqueness of the solutions for the closed-loop system, the study of the problem with non constant (dynamic) contact angles, the study of the output feedback stabilization problem, the construction of appropriate numerical schemes and the derivation of stability estimates in stronger spatial norms. Concerning the output feedback stabilization problem there are many interesting studies in the literature that may contribute, such as \cite{Y2} which suggests a finite-dimensional observer control of the (1-D) heat equation under Neumann actuation.

\end{document}